\documentclass[a4paper,11pt]{article}
\usepackage{amssymb}
\usepackage{amsmath}
\usepackage{amsthm}
\usepackage{graphicx}
\usepackage{cite}
\usepackage[square,numbers,sort&compress]{natbib}
\newtheorem{thm}{Theorem}[section]
\newtheorem{cor}{Corollary}[section]
\newtheorem{lem}{Lemma}[section]

\theoremstyle{definition}
\newtheorem{defn}{Definition}[section]
\theoremstyle{remark}
\newtheorem{rem}{Remark}[section]

\numberwithin{equation}{section}

\begin{document}
	\begin{center}
		{\Large \bf On Fractional derivative of Hurwitz Zeta function and Jacobi Theta function}
	\end{center}
	\begin{center}
		{\normalsize \bf $^{(a)}$Ashish Mor, $^{(a)}$Surbhi Gupta\footnote{ corresponding author (sgupta11@amity.edu)} and $^{(b)}$Manju Kashyap}
	\end{center}
	\begin{center}
		
		{\normalsize $^{(a)}$Department of Applied Mathematics, Amity University\\ Noida-201313, India.\\
			$^{(b)}$Department of Applied Mathematics, Galgotias College of Engineering and Technology\\ Gr. Noida-201310, India}
	\end{center}
	
	\begin{abstract}
		This research paper focuses on exploring two Complex-valued function's fractional derivative, specifically the Hurwitz Zeta function and Jacobi theta function. The study is based on the Complex Generalization of Grünwald-Letnikov Fractional derivative which adheres to the generalized version of the Leibniz rule. Within this paper, we present and establish an identity that serves as a Generalization of the Hurwitz Zeta function's Functional Equation. Additionally, we derive the Jacobi Theta function's Functional Equation, accompanied by a comprehensive examination of the properties associated with the Fractional derivative of these two functions.
	\end{abstract}
	
	
	\section{Introduction}
	In recent years, numerous researchers have dedicated their studies to the exploration of fractional derivatives in the context of complex functions. Ortigueira, in particular, extended fractional calculus for the complex variable in the complex plane, aiming to uncover potential signal processing applications \cite{r1,r2}. Similarly, Owa demonstrated the extension of Fractional derivatives to analytical functions within the unit circle\cite{r3}.
	
	The realm of Complex-valued functions encompasses a wide array of 	Special functions\cite{r4}. Within this expansive collection, the Zeta function family holds a significant position as a crucial component of analytic number theory \cite{r5}. Its presence extends to various branches of mathematics, including both pure and applied domains.
	
	The Riemann Zeta function ($\zeta(s)$)\cite{r5} and its generalized counterpart Hurwitz Zeta function ($\zeta(s,a)$)\cite{r6} hold considerable importance across various Scientific disciplines. These functions play pivotal roles in areas such as Number theory, Mathematical Physics, and Signal Processing, making them indispensable tools for research and analysis.
	
	The Riemann Zeta function $(\zeta(s))$ possesses a Fractional derivative that enables several generalizations. Notable among these are the Dirichlet series \cite{r8} and the Hurwitz Zeta function's $(\zeta(s,a))$ fractional derivative\cite{r6}. This paper exclusively focuses on the analysis of the Hurwitz Zeta function in the context of fractional calculus, exploring its properties and applications within the realm of Fractional calculus. For further investigations into the interplay between the prime numbers distribution and the Fractional Calculus of the Riemann Zeta function, additional references can be found \cite{r9,r10,r11}.
	
	Among the diverse category of complex-valued functions, the Jacobi theta function  ($\theta(s)$)\cite{r12}, also known as the classical theta function or simply theta function, holds a prominent position and finds numerous applications in number theory. This function is associated with renowned $Theorem's$ and $Identities$ in Mathematics, including the well known $The$ $Two-Squares$ $Theorem$\cite{r13} which states that: Let $n$ be a positive integer and if we want to represent $n$ as the sum of two squares then it only happens if and only if the prime factorization of $n$ contains an even exponent for each prime $p_j$ that has the form $4k + 3$ (where $k$ is an integer). Additionally, the $Jacobi$ $Triple$ $Product$ $Identity$ \cite{r14} is another significant result related to the Jacobi theta function.
	
	In particular, theta function plays an important role in theory of Elliptic curves\cite{r15} which is an algebraic Curve of genus one having property that it is also smooth and projective and can be expressed as the Weierstrass Equation defined as $y^2=x^3+Ax+B$ where $A$ and $B$ are constants, for an Elliptic curve and also during the study of Modular forms\cite{r16} which are Complex analytic functions on the upper half-plane satisfying a certain type of functional equation with respect to the group action of the modular group\cite{r17} and the Eisenstein series\cite{r18} is the most well-known illustration of modular form.
	
	The Fractional analysis of this theta function is one of the subject of this paper along with a relation between Jacobi theta function and the Fractional derivative of Riemann Zeta function.
	
	In the subsequent discussion, we delve into the computation of the Fractional derivatives of the Hurwitz Zeta function and the Jacobi theta function. We derive an identity that serves as a Generalization of the Fractional Functional Equation for the Hurwitz Zeta function. Additionally, we provide Jacobi theta function's fractional functional equation and explore the properties associated with the Hurwitz Zeta function's fractional functional equation. Remarkably, from these findings, we obtain the Fractional Functional Equation for the Riemann Zeta function as a special case.
	\section{Notations and preliminary results}
	This section introduces certain terminology and definitions that will be utilized in this paper. These foundational concepts establish a common understanding and serve as a basis for the subsequent discussions.
	
	Let $\alpha\in\mathbb{R}$, $n\in\mathbb{N}$ and $s\in\mathbb{C}$. The $k$th falling factorial of $\alpha$\cite{r7} is denoted by $\alpha^{\underline{k}}$.
	
	The Riemann Zeta function's definition\cite{r5}
	\begin{equation}\label{eqn1}
		\zeta(s) = \sum_{n=1}^{\infty} \frac{1}{n^s},  
	\end{equation}
	holds for  $s\in\mathbb{C}$ with $Re(s)>1$, allows a number of generalizations. In particular, Dirichlet series and the Hurwitz Zeta function described\cite{r6} by
	\begin{equation}\label{eqn2}
		F(s)=\sum_{n=1}^{\infty} \frac{f(n)}{n^s}, f:\mathbb{N}	\rightarrow \mathbb{C}
	\end{equation}
	and
	\begin{equation}\label{eqn3}
		\zeta(s,a) = \sum_{m=0}^{\infty} \frac{1}{(m+a)^s},
	\end{equation}
	both holds for $s\in\mathbb{C}$ with $Re(s)>1$ and $a\in\mathbb{R}:0<a\leq1$.
	
	From the above Equations $\ref{eqn2}$ and $\ref{eqn3}$, we can notice that  $$F(s;f(n)=1)=\zeta(s),$$
	and $$\zeta(s,1)=\zeta(s).$$
	We refer \cite{r8} to reader for studying Dirichlet series Fractional Calculus and we will not discuss upon this topic in this paper. Notably, Equations $\ref{eqn1}$ and $\ref{eqn3}$ display the same behavior in terms of analyticity. In fact, these Zeta functions for $s\in\mathbb{C}$ with $Re(s)>1$, converges absolutely. Additionally, Equations $\ref{eqn1}$ and $\ref{eqn3}$ could be continued on the entire complex plane to a holomorphic function, except an isolated singularity (pole) at s=1 of order 1 having residue 1 given\cite{r5} by 
	\begin{equation}\label{eqn4}
		\zeta(s)=2(2\pi)^{s-1}\Gamma(1-s)sin\left(\frac{\pi s}{2}\right)\zeta(1-s) 
	\end{equation}
	holds for all $s\in\mathbb{C}$.
	
	Further $\zeta(s,a)$ satisfies\cite{r19}
	\begin{equation}\label{eqn5}
		\zeta(s,a)=\frac{2\Gamma(1-s)}{(2\pi)^{1-s}}\left(sin\left(\frac{\pi s}{2}\right)\sum_{q=1}^{\infty} \frac{cos(2\pi qa)}{q^{1-s}} + cos\left(\frac{\pi s}{2}\right)\sum_{q=1}^{\infty} \frac{sin(2\pi qa)}{q^{1-s}}\right)
	\end{equation}	
	holds for $Re(s)<0$ and $0<a\leq1$.
	\begin{defn} Let $k\in\mathbb{N}$ and $\alpha$ be any real number. The symbol $\binom{\alpha}{k}$ which is known as Generalized binomial coefficient is defined\cite{r6} as
		\begin{equation}\label{eqn6}
			\binom{\alpha}{k}=\frac{\alpha^{\underline{k}}}{k!}=\frac{\alpha(\alpha-1).....(\alpha-k+1)}{k!}
		\end{equation}
	\end{defn}
	\begin{defn} The forward Grünwald-Letnikov fractional derivative of $f$ is defined\cite{r6} as 
		\begin{equation}\label{eqn7}
			D_{f}^{\alpha}f(s)=	\lim_{l \to 0^{+}} \frac{\sum_{m=0}^{\infty}{\binom{\alpha}{m}(-1)^m f(s-ml)}}{l^{\alpha}},
		\end{equation}
		where $\alpha\in\mathbb{R}$ and $f$ is a function which inside the region $D\subseteq \mathbb{C}$, is holomorphic and it has a property that it is continuous on its contour denoted by $C_d$.
	\end{defn}
	\begin{rem}
		In Equation $\ref{eqn7}$, the selection of the fractional derivative is dependent on a pertinent condition. Furthermore, this derivative actually adheres to the generalised Leibniz rule:
		
		Let two functions be represented by $f$ and $g$ then
		\begin{equation}\label{eqn8}
			D_{f}^{\alpha}(f(s)g(s))=\sum_{k=0}^{\infty} \binom{\alpha}{k} f^{(k)}g^{(\alpha-k)},
		\end{equation}
		where $s\in\mathbb{C}$ and $\alpha\in\mathbb{R}_{>0}\backslash\mathbb{N},$ with the constraint that $f$ must be holomorphic inside the region $D\subseteq \mathbb{C}$.
		
	\end{rem}
	The fractional derivative in Equation $\ref{eqn7}$ covers a very basic role in Fractional Calculus \cite{r20} of functions which are holomorphic.
	
	Further, Equation $\ref{eqn8}$ implies\cite{r7} that
	\begin{multline}\label{eqn9}
		\zeta^{\alpha}(s)=2{(2\pi)}^{s-1}e^{\dot{\iota}\pi\alpha}\sum_{r=0}^{\infty} \sum_{k=0}^{\infty} \sum_{l=0}^{\infty} A^{\alpha}_{r,k,l} \zeta^{r}(1-s)(-\frac{\pi}{2})^{k}sin\left(\frac{\pi (s+k)}{2}\right)\\ 
		\times \frac{\Gamma^{(l)}(1-s)}{log^{r+k+l-\alpha}(2\pi)},
	\end{multline}
	where $A^{\alpha}_{r,k,l}=\frac{\alpha^{\underline{r+k+l}}}{r!k!l!}.$ 
	
	Furthermore, Equation $\ref{eqn9}$ has been Generalized to the Functional Equation of the Hurwitz zeta function given\cite{r6} by
	\begin{multline}\label{eqn10}
		\zeta^{\alpha}(s,\frac{p}{q})=2{(2 \pi q)}^{s-1}e^{\dot{\iota}\pi\alpha}\sum_{m=0}^{\infty} \sum_{k=0}^{\infty} \sum_{l=0}^{\infty}
		A^{\alpha}_{m,k,l}\frac{\Gamma^{(m)}(1-s)}{log^{m+k+l-\alpha}(2\pi q)} (-\frac{\pi}{2})^{k}\\
		\times \sum_{h=1}^{q} sin\left(\frac{\pi (s+k)}{2}+\frac{2\pi hp}{q}\right)	\zeta^{l}\left(1-s,\frac{h}{q}\right),	
	\end{multline}
	where $p,q\in\mathbb{Z}:1\leq p\leq q$, $\alpha\in\mathbb{R}_{>0}\backslash\mathbb{N}$ and $A^{\alpha}_{m,k,l}=\frac{\alpha^{\underline{m+k+l}}}{m!k!l!}.$ 
	\begin{rem}
		In this paper, we extend equation $\ref{eqn10}$ to any Real number $a:0<a\leq1$ instead of just a Rational number $\frac{p}{q}$ where $p,q\in\mathbb{Z}:1\leq p\leq q$ by providing an identity for Fractional derivative of Hurwitz Zeta function namely $\zeta^{\alpha}(s,a)$ which holds for all $s\in\mathbb{C}$ with $Re(s)<0$.
	\end{rem}
	
	Notice that the Equation $\ref{eqn10}$ can readily be rewritten as an algebraic combinations of trigonometric functions specifically the sine and cosine functions \cite{r6} that lowers the expensive computational cost.
	
	The Jacobi theta function, also known as the classical theta function or simply theta function is the function on right half plane($Re(s)>0$) defined\cite{r12} by
	\begin{equation}\label{11}
		\theta(s)=\sum_{n\in\mathbb{Z}} e^{-\pi n^{2}s}, Re(s)>0.
	\end{equation}
	\begin{rem}
		For the remainder of our discussion, we'll simply refer to this function as $theta$ $function$.
	\end{rem}
	The theta function is a function, which on the right half plane is holomorphic and it also obeys the Functional Equation defined\cite{r21} by
	\begin{equation}\label{eqn12}
		\theta(s)=\frac{1}{\sqrt{s}}\theta\left(\frac{1}{s}\right), Re(s)>0
	\end{equation}
	This paper provides a Fractional Functional Equation for this theta function and also establishes a relation between Riemann Zeta function's fractional derivative and Jacobi theta function.
	\section{Statement and Description of Results}
	In this section, first we discuss the properties of the Fractional derivative of Hurwitz Zeta function in $\boldsymbol{(\ref{subsec1})}$ and then in $\boldsymbol{(\ref{subsec2})}$ we state and prove an identity which generalizes the Fractional Functional Equation of Hurwitz Zeta function and after this we will prove the simplified form of this Functional Equation in $\boldsymbol{(\ref{subsec3})}$ and then in $\boldsymbol{(\ref{subsec4})}$ we first state and then prove the Jacobi theta function's fractional functional equation and at last in $\boldsymbol{(\ref{subsec5})}$ we establish a relation between Riemann Zeta function's fractional derivative and Jacobi theta function.
	\subsection{Properties of $\zeta^{\alpha}(s,a)$}\label{subsec1}
	\begin{lem}\label{lem1}
		Let $P$ and $Q$ represent two functions using the Dirichlet series, that is
		\begin{equation}\label{eqn13}
			P(s)=\sum_{k=1}^{\infty} \frac{p(k)}{k^s}, s\in\mathbb{C} with Re(s)>a.
		\end{equation}
		\begin{equation}\label{eqn14}
			Q(s)=\sum_{k=1}^{\infty} \frac{q(k)}{k^s}, s\in\mathbb{C} with Re(s)>b.
		\end{equation}
		According to these assumptions, in the region (half-plane) where both series $(\ref{eqn13})$ and $(\ref{eqn14})$ above converges absolutely, then
		\begin{equation}\label{eqn15}
			P(s)Q(s)=\sum_{k=1}^{\infty} \frac{p(k)*q(k)}{k^s},
		\end{equation}
		in which $p*q$, the Dirichlet convolution, is defined [\cite{r5}, chap. 2] by
		\begin{equation}\label{eqn16}
			p(k)*q(k)=\sum_{d|k} p(k)q\left(\frac{k}{d}\right).
		\end{equation}
		For the proof of Lemma $\ref{lem1}$ see [\cite{r5}, p. 228].
	\end{lem}
	\begin{thm}\label{thm1}
		Let $s\in\mathbb{C}$ and $\alpha\in\mathbb{R}\backslash\mathbb{Z}$ then the product between Hurwitz Zeta function's fractional derivative of order $\alpha$ and Hurwitz Zeta function is given by
		\begin{equation}\label{eqn17}
			\zeta^{\alpha}(s,a)\zeta(s,a)=e^{\dot{\iota}\pi\alpha}\sum_{k=0}^{\infty} \sum_{d|k} \frac{log^{\alpha}(d+a)}{(k+a)^{s}},
		\end{equation}
		in the region (half-plane) $Re(s)>1+\alpha$. 
	\end{thm}
	\begin{proof}
		From\cite{r6}, we know that for $\alpha\in\mathbb{R}\backslash\mathbb{Z}$, we have
		\begin{equation}\label{eqn18}
			\zeta^{\alpha}(s,a)=e^{\dot{\iota}\pi\alpha} \sum_{k=0}^{\infty} \frac{log^{\alpha}(k+a)}{(k+a)^s}
		\end{equation}
		and from the Lemma $\ref{lem1},$ we have 
		\begin{equation}\label{eqn19}
			\zeta^{\alpha}(s,a)\zeta(s,a)=e^{\dot{\iota}\pi\alpha}\sum_{k=0}^{\infty} \frac{log^{\alpha}(k+a)}{(k+a)^s}\sum_{k=0}^{\infty} \frac{1}{(k+a)^s}=e^{\dot{\iota}\pi\alpha}\sum_{k=0}^{\infty} \frac{log^{\alpha}(k+a)*1}{(k+a)^s},
		\end{equation} 
		where
		\begin{equation}\label{eqn20}
			log^{\alpha}(k+a)*1=\sum_{d|k} log^{\alpha}(d+a)
		\end{equation}
		The argument follows since Equation $\ref{eqn18}$ converges for $Re(s)>(1+\alpha)$\cite{r22}.
	\end{proof}
	\subsection{Generalization of Functional Equation for $\zeta^{\alpha}(s,a)$}\label{subsec2}
	\begin{lem}\label{lem2}
		Let $g:	D\subseteq\mathbb{C}\rightarrow\mathbb{C}$ represents a function such that the ratio in Equation $\ref{eqn7}$ which is known as fractional incremental ratio is uniformly convergent inside D then
		\begin{equation}\label{eqn21}
			D^{\alpha}_g(s)\xrightarrow{\alpha\rightarrow k^{-}}g^{(k)}(s)
		\end{equation}
		\begin{equation}\label{eqn22}
			D^{\alpha}_g(s)\xrightarrow{\alpha\rightarrow (k-1)^{+}}g^{(k-1)}(s)
		\end{equation}
		for any $\alpha\in\mathbb{R}\backslash\mathbb{Z}:\lfloor\alpha\rfloor=k-1$.
		
		Proof of Lemma $\ref{lem2}$ can be found in \cite{r7}.
		\begin{rem}
			Equations $\ref{eqn1},\ref{eqn2},\ref{eqn3}$ fulfill all hypotheses of Equations $\ref{eqn21}$ and $\ref{eqn22}$.
		\end{rem}
	\end{lem}
	
	\begin{thm}\label{thm2}
		Let $a\in\mathbb{R}:0<a\leq1$ and $\alpha\in\mathbb{R}_{>0}\backslash\mathbb{N}$. Then for any $s\in\mathbb{C}$ with $Re(s)<0$,
		\begin{multline}\label{eqn23}
			\zeta^{\alpha}(s,a)=2{(2 \pi)}^{s-1}e^{\dot{\iota}\pi\alpha}\sum_{r=0}^{\infty} \sum_{k=0}^{\infty} \sum_{l=0}^{\infty}
			A^{\alpha}_{r,k,l}\frac{\Gamma^{(r)}(1-s)}{log^{r+k+l-\alpha}(2\pi )} (-\frac{\pi}{2})^{k}\\ \times \left[sin\left(\frac{\pi (s+k)}{2}\right)\sum_{q=1}^{\infty} \frac{(-logq)^{l}cos(2\pi qa)}{q^{1-s}} + cos\left(\frac{\pi (s+k)}{2}\right)\sum_{q=1}^{\infty} \frac{(-logq)^{l}sin(2\pi qa)}{q^{1-s}}\right],	
		\end{multline}
		where $A^{\alpha}_{r,k,l}=\frac{\alpha^{\underline{r+k+l}}}{r!k!l!}.$
		\begin{proof}
			First from Equation $\ref{eqn5}$ we know that $\zeta(s,a)$ satisfies
			$$\zeta(s,a)=\frac{2\Gamma(1-s)}{(2\pi)^{1-s}}\left(sin\left(\frac{\pi s}{2}\right)\sum_{q=1}^{\infty} \frac{cos(2\pi qa)}{q^{1-s}} + cos\left(\frac{\pi s}{2}\right)\sum_{q=1}^{\infty} \frac{sin(2\pi qa)}{q^{1-s}}\right)$$
			holds for $s\in\mathbb{C}$ with $Re(s)<0$ and $0<a\leq1$.
			
			Now take the Fractional derivative of $\zeta(s,a)$ of order $\alpha$ and use Generalized Leibniz rule i.e. Equation $\ref{eqn8}$ to Right hand side of the above Equation. This implies
			\begin{multline}\label{eqn24}
				\zeta^{\alpha}(s,a)=2\sum_{r=0}^{\infty} \binom{\alpha}{r} \frac{d^r}{ds^r}(\Gamma(1-s))\\\times D^{\alpha-r}_f\left((2\pi)^{s-1}sin\left(\frac{\pi s}{2}\right)\sum_{q=1}^{\infty}\frac{cos(2\pi qa)}{q^{1-s}}+(2\pi)^{s-1}cos\left(\frac{\pi s}{2}\right)\sum_{q=1}^{\infty}\frac{sin(2\pi qa)}{q^{1-s}}\right)
			\end{multline}
			This further implies
			\begin{multline}\label{eqn25}
				\zeta^{\alpha}(s,a)=2\sum_{r=0}^{\infty} \binom{\alpha}{r} \frac{d^r}{ds^r}(\Gamma(1-s))\\\times\left[\sum_{k=0}^{\infty}\binom{\alpha-r}{k}\frac{d^k}{ds^k}sin\left(\frac{\pi s}{2}\right)D^{\alpha-r-k}_f\left((2\pi)^{s-1}\sum_{q=1}^{\infty}\frac{cos(2\pi qa)}{q^{1-s}}\right)\right]\\+2\sum_{r=0}^{\infty} \binom{\alpha}{r} \frac{d^r}{ds^r}(\Gamma(1-s))\\\times\left[\sum_{k=0}^{\infty}\binom{\alpha-r}{k}\frac{d^k}{ds^k}cos\left(\frac{\pi s}{2}\right)D^{\alpha-r-k}_f\left((2\pi)^{s-1}\sum_{q=1}^{\infty}\frac{sin(2\pi qa)}{q^{1-s}}\right)\right]
			\end{multline}
			This gives
			\begin{multline}\label{eqn26}
				\zeta^{\alpha}(s,a)=2\sum_{r=0}^{\infty} \binom{\alpha}{r} \frac{d^r}{ds^r}(\Gamma(1-s))\\\times\left[\sum_{k=0}^{\infty}\binom{\alpha-r}{k}\frac{d^k}{ds^k}sin\left(\frac{\pi s}{2}\right)\sum_{l=0}^{\infty}\binom{\alpha-r-k}{l}\frac{d^l}{ds^l}\left(\sum_{q=1}^{\infty}\frac{cos(2\pi qa)}{q^{1-s}}\right)\right]\\\times \left(D^{\alpha-r-k-l}_f(2\pi)^{s-1}\right)+2\sum_{r=0}^{\infty} \binom{\alpha}{r} \frac{d^r}{ds^r}(\Gamma(1-s))\\\times \left[\sum_{k=0}^{\infty}\binom{\alpha-r}{k}\frac{d^k}{ds^k}cos\left(\frac{\pi s}{2}\right)\sum_{l=0}^{\infty}\binom{\alpha-r-k}{l}\frac{d^l}{ds^l}\left(\sum_{q=1}^{\infty}\frac{sin(2\pi qa)}{q^{1-s}}\right)\right]\\\times \left(D^{\alpha-r-k-l}_f(2\pi)^{s-1}\right)
			\end{multline}
			Immediately, we see that
			\begin{equation}\label{eqn27}
				\frac{d^r}{ds^r}(\Gamma(1-s))=e^{\dot{\iota}\pi r}\Gamma^{(r)}(1-s),	
			\end{equation}
			\begin{equation}\label{eqn28}
				\frac{d^k}{ds^k}sin\left(\frac{\pi s}{2}\right)=\left(\frac{\pi}{2}\right)^ksin\left(\frac{\pi (s+k)}{2}\right),
			\end{equation}
			\begin{equation}\label{eqn29}
				\frac{d^k}{ds^k}cos\left(\frac{\pi s}{2}\right)=\left(\frac{\pi}{2}\right)^kcos\left(\frac{\pi (s+k)}{2}\right),
			\end{equation}
			\begin{equation}\label{eqn30}
				\frac{d^l}{ds^l}\left(\sum_{q=1}^{\infty}\frac{cos(2\pi qa)}{q^{1-s}}\right)=\sum_{q=1}^{\infty} \frac{(logq)^{l}cos(2\pi qa)}{q^{1-s}},
			\end{equation}
			\begin{equation}\label{eqn31}
				\frac{d^l}{ds^l}\left(\sum_{q=1}^{\infty}\frac{sin(2\pi qa)}{q^{1-s}}\right)=\sum_{q=1}^{\infty} \frac{(logq)^{l}sin(2\pi qa)}{q^{1-s}},
			\end{equation}
			In addition, given Corollary 4 in \cite{r7}, we have
			\begin{equation}\label{eqn32}
				D^{\alpha}_f (2\pi)^s=(2\pi)^se^{\dot{\iota}\pi\alpha}log^{\alpha}(2\pi),
			\end{equation}
			and so is
			\begin{equation}\label{eqn33}
				D^{\alpha-r-k-l}_f (2\pi)^{s-1}=(2\pi)^{s-1}e^{\dot{\iota}\pi(\alpha-r-k-l)}log^{\alpha-r-k-l}(2\pi),
			\end{equation}
			\begin{equation}\label{eqn34}
				A^{\alpha}_{r,k,l}=\binom{\alpha}{r}\binom{\alpha-r}{k}\binom{\alpha-r-k}{l}=\frac{\alpha^{\underline{r+k+l}}}{r!k!l!},
			\end{equation}
			Putting all these above Equations together in Equation $\ref{eqn26}$, we get
			\begin{multline}\label{eqn35}
				\zeta^{\alpha}(s,a)=2\sum_{r=0}^{\infty} \binom{\alpha}{r} e^{\dot{\iota}\pi r}(\Gamma^{(r)}(1-s))\\\times\left[\sum_{k=0}^{\infty}\binom{\alpha-r}{k}\left(\frac{\pi}{2}\right)^ksin\left(\frac{\pi (s+k)}{2}\right)\sum_{l=0}^{\infty}\binom{\alpha-r-k}{l}\sum_{q=1}^{\infty} \frac{(logq)^{l}cos(2\pi qa)}{q^{1-s}}\right]\\\times\left((2\pi)^{s-1}e^{\dot{\iota}\pi(\alpha-r-k-l)}log^{\alpha-r-k-l}(2\pi)\right)
				\\+2\sum_{r=0}^{\infty} \binom{\alpha}{r} e^{\dot{\iota}\pi r}(\Gamma^{(r)}(1-s))\\\times \left[\sum_{k=0}^{\infty}\binom{\alpha-r}{k}\left(\frac{\pi}{2}\right)^kcos\left(\frac{\pi (s+k)}{2}\right)\sum_{l=0}^{\infty}\binom{\alpha-r-k}{l}\sum_{q=1}^{\infty} \frac{(logq)^{l}sin(2\pi qa)}{q^{1-s}}\right]\\\times\left((2\pi)^{s-1}e^{\dot{\iota}\pi(\alpha-r-k-l)}log^{\alpha-r-k-l}(2\pi)\right)
			\end{multline}
			Now on simplifying Equation $\ref{eqn35}$, we get Theorem $\ref{thm2}$ which  holds for all $s\in\mathbb{C}$ with $Re(s)<0$ and $0<a\leq1$.
		\end{proof}
	\end{thm}
	\begin{rem}
		Now, let's focus only on the consistency of Equation $\ref{eqn23}$. Note that by Equations $\ref{eqn21}$ and $\ref{eqn22}$,
		\begin{equation}\label{eqn36}
			g^{(\alpha)}(s)\rightarrow g^{(0)}(s)=g(s).
		\end{equation}
		as $\alpha\rightarrow0^{+}$.
		
		Specifically, all of Equation $\ref{eqn21}$ and $\ref{eqn22}$'s hypotheses are satisfied by the Hurwitz Zeta function, and thus
		\begin{equation}\label{eqn37}
			\zeta^{(\alpha)}(s,a)\rightarrow \zeta(s,a).
		\end{equation}
		as $\alpha\rightarrow0^{+}$.
		
		Theorem $\ref{thm2}$'s proof can, of course, be read backwards until Equation $\ref{eqn24}$. As a result, Equation $\ref{eqn24}$'s right-hand side converges to that of Equation $\ref{eqn5}$
		as $\alpha\rightarrow0^{+}$.
		Therefore, Equation $\ref{eqn23}$ agrees with the definitions and formulations of Zeta functions.
	\end{rem}
	\begin{rem}
		If we put $a=1$ in Theorem $\ref{thm2}$, we get Fractional Functional Equation for Riemann Zeta function which holds for all $s\in\mathbb{C}$ with $Re(s)<0$ namely Equation $\ref{eqn9}$ as a special case of our Theorem  $\ref{thm2}$.	
	\end{rem}
	As one might expect, it becomes practically hard to compute with and use the identity for $\zeta^{\alpha}(s,a)$ from Theorem $\ref{thm2}$. The triple sum we shall meet will cause even a reduced version of $\zeta^{\alpha}(s,a)$ to approach a numerical value relatively slowly. Because of this, we are interested in improving the formula for $\zeta^{\alpha}(s,a)$ by computing the derivative of order $\alpha$ in a new way, which we accomplish as follows:
	\subsection{Simplified form of $\zeta^{\alpha}(s,a)$}\label{subsec3}
	\begin{thm}\label{thm3}
		Let $s\in\mathbb{C}$ with $Re(s)<0$ and $\alpha\in\mathbb{R}_{>0}\backslash\mathbb{N},$ then we have
		\begin{multline}\label{eqn38}
			\zeta^{\alpha}(s,a)=\dot{\iota}(2\pi)^{s-1}\sum_{h=0}^{\infty} \binom{\alpha}{h}\left[\frac{d^h}{ds^h}\left(\sum_{q=1}^{\infty} \frac{cos(2\pi qa)}{q^{1-s}}+\dot{\iota}\sum_{q=1}^{\infty}\frac{sin(2\pi qa)}{q^{1-s}}\right)\Gamma(1-s)\right]\\\times\left[e^{\frac{-\dot{\iota}\pi s}{2}}\bar{\tau}^{\alpha-h}-e^{\frac{\dot{\iota}\pi s}{2}}\tau^{\alpha-h}\right],
		\end{multline}
		where $\tau=log2\pi+\frac{\dot{\iota}\pi}{2}.$
		\begin{proof}
			The intuation for this proof comes from rewriting the terms  $sin\frac{\pi s}{2}$ and $cos\frac{\pi s}{2}$ in their Euler forms in the Equation $\ref{eqn5}$. So from Equation $\ref{eqn5}$, for $s\in\mathbb{C}$ with $Re(s)<0$, we have
			$$\zeta(s,a)=2(2\pi)^{s-1}\Gamma(1-s)\left(sin\left(\frac{\pi s}{2}\right)\sum_{q=1}^{\infty} \frac{cos(2\pi qa)}{q^{1-s}} + cos\left(\frac{\pi s}{2}\right)\sum_{q=1}^{\infty} \frac{sin(2\pi qa)}{q^{1-s}}\right),$$
			Now write $$sin\left(\frac{\pi s}{2}\right)=\frac{e^\frac{\dot{\iota}\pi s}{2}-e^\frac{-\dot{\iota}\pi s}{2}}{2\dot{\iota}}$$ and $$cos\left(\frac{\pi s}{2}\right)=\frac{e^\frac{\dot{\iota}\pi s}{2}+e^\frac{-\dot{\iota}\pi s}{2}}{2}$$ and by putting these in the above Equation, we get
			\begin{multline}\label{eqn39}
				\zeta(s,a)=2e^{(s-1)log(2\pi)}\Gamma(1-s)\\\times\left(\frac{e^\frac{\dot{\iota}\pi s}{2}-e^\frac{-\dot{\iota}\pi s}{2}}{2\dot{\iota}}\sum_{q=1}^{\infty} \frac{cos(2\pi qa)}{q^{1-s}} + \frac{e^\frac{\dot{\iota}\pi s}{2}+e^\frac{-\dot{\iota}\pi s}{2}}{2}\sum_{q=1}^{\infty} \frac{sin(2\pi qa)}{q^{1-s}}\right),
			\end{multline}
			This further simplifies to
			\begin{multline}\label{eqn40}
				\zeta(s,a)=\left[e^{(s-1)(log2\pi+{\frac{\dot{\iota}\pi}{2})}}+e^{(s-1)(log2\pi-\frac{\dot{\iota}\pi}{2})}\right]\Gamma(1-s)\\\times\left(\sum_{q=1}^{\infty} \frac{cos(2\pi qa)}{q^{1-s}}+\dot{\iota}\sum_{q=1}^{\infty}\frac{sin(2\pi qa)}{q^{1-s}}\right)
			\end{multline}
			Now set $$\tau=log2\pi+\frac{\dot{\iota}\pi}{2} , g(s)=\left(\sum_{q=1}^{\infty} \frac{cos(2\pi qa)}{q^{1-s}}+\dot{\iota}\sum_{q=1}^{\infty}\frac{sin(2\pi qa)}{q^{1-s}}\right)$$
			and $$\psi(s,\tau)=e^{(s-1)\tau}\Gamma(1-s)g(s)$$
			This implies
			
			$$\zeta(s,a)=\psi(s,\tau)+\psi(s,\bar{\tau})$$
			Now use the Grünwald-Letnikov derivative to $\zeta(s,a)$ and get
			$$\zeta^{\alpha}(s,a)=\psi^{\alpha}(s,\tau)+\psi^{\alpha}(s,\bar{\tau})$$
			Next, we calculate each Fractional derivative of $\psi$ separately, and finally we add them all up.
			\begin{align*}
				\psi^{\alpha}(s,\tau)
				&=D^{\alpha}_f\left(e^{(s-1)\tau}\Gamma(1-s)g(s)\right)\\
				&=\sum_{h=0}^{\infty} \binom{\alpha}{h}\left[\frac{d^h}{ds^h}g(s)\Gamma(1-s)\right]D^{\alpha-h}_f e^{(s-1)\tau}
			\end{align*}
			where $D^{\alpha-h}_f e^{(s-1)\tau}=e^{(s-1)\tau}\tau^{\alpha-h}$ (by Equation $\ref{eqn7}$).
			
			Consequently we arrive at
			\begin{align*}
				\psi^{\alpha}(s,\tau)
				&=\sum_{h=0}^{\infty} \binom{\alpha}{h}\left[\frac{d^h}{ds^h}g(s)\Gamma(1-s)\right]D^{\alpha-h}_f e^{(s-1)\tau}\\
				&=\sum_{h=0}^{\infty} \binom{\alpha}{h}\left[\frac{d^h}{ds^h}g(s)\Gamma(1-s)\right]e^{(s-1)\tau}\tau^{\alpha-h}
			\end{align*}
			Similarly we get
			$$\psi^{\alpha}(s,\bar{\tau})=\sum_{h=0}^{\infty} \binom{\alpha}{h}\left[\frac{d^h}{ds^h}g(s)\Gamma(1-s)\right]e^{(s-1)\bar{\tau}}\bar{\tau}^{\alpha-h}$$
			Now by combining the two terms $\psi^{\alpha}(s,\tau)$ and $\psi^{\alpha}(s,\bar{\tau})$ and realizing that
			$$e^{-\tau}=-e^{\frac{\dot{\iota}\pi}{2}}\frac{1}{2\pi}$$ and $$e^{-\bar{\tau}}=e^{\frac{\dot{\iota}\pi}{2}}\frac{1}{2\pi}$$
			we finally get
			\begin{align*}
				\zeta^{\alpha}(s,a)
				&=\sum_{h=0}^{\infty} \binom{\alpha}{h}\left[\frac{d^h}{ds^h}g(s)\Gamma(1-s)\right]\left[e^{(s-1)\bar{\tau}}\bar{\tau}^{\alpha-h}+e^{(s-1)\tau}\tau^{\alpha-h}\right]\\
				&=\sum_{h=0}^{\infty} \binom{\alpha}{h}\left[\frac{d^h}{ds^h}g(s)\Gamma(1-s)\right]\left[e^{\frac{\dot{\iota}\pi}{2}}\frac{1}{2\pi}e^{s\bar{\tau}}\bar{\tau}^{\alpha-h}-e^{\frac{\dot{\iota}\pi}{2}}\frac{1}{2\pi}e^{s\tau}\tau^{\alpha-h}\right]\\
				&=\frac{\dot{\iota}}{2\pi}\sum_{h=0}^{\infty} \binom{\alpha}{h}\left[\frac{d^h}{ds^h}g(s)\Gamma(1-s)\right]\left[(2\pi)^s e^{\frac{-\dot{\iota}\pi s}{2}}\bar{\tau}^{\alpha-h}-(2\pi)^s e^{\frac{\dot{\iota}\pi s}{2}}\tau^{\alpha-h}\right]\\
				&=\dot{\iota}(2\pi)^{s-1}\sum_{h=0}^{\infty}\binom{\alpha}{h}\left[\frac{d^h}{ds^h}g(s)\Gamma(1-s)\right]\left[e^{\frac{-\dot{\iota}\pi s}{2}}\bar{\tau}^{\alpha-h}-e^{\frac{\dot{\iota}\pi s}{2}}\tau^{\alpha-h}\right]
			\end{align*}
			On putting the value of $g(s)$ above, we get Theorem $\ref{thm3}$.
		\end{proof}
	\end{thm}
	Now by Theorem $\ref{thm3}$ we will prove Simplified form of Equation $\ref{eqn9}$ which is Simplified form of Functional Equation of $\zeta^{\alpha}(s)$.
	\begin{cor}
		Let $s\in\mathbb{C}$ with $Re(s)<0$ and $\alpha\in\mathbb{R}_{>0}\backslash\mathbb{N},$ then we have
		\begin{multline}\label{eqn41}
			\zeta^{\alpha}(s)=\dot{\iota}(2\pi)^{s-1}\sum_{h=0}^{\infty} \binom{\alpha}{h}\left[\frac{d^h}{ds^h}\zeta(1-s)\Gamma(1-s)\right]\left[e^{\frac{-\dot{\iota}\pi s}{2}}\bar{\tau}^{\alpha-h}-e^{\frac{\dot{\iota}\pi s}{2}}\tau^{\alpha-h}\right],
		\end{multline}
		where $\tau=log2\pi+\frac{\dot{\iota}\pi}{2}.$
		\begin{proof}
			Putting $a=1$ in Theorem $\ref{thm3}$, we get the desired result because we know that $\zeta^{\alpha}(s,1)=\zeta^{\alpha}(s)$.
		\end{proof}
	\end{cor}
	In addition, Equation $\ref{eqn23}$ can be  rewritten as an algebraic combinations of trigonometric functions specifically the sine and cosine functions that lowers the expensive computational cost as follows:
	\begin{thm}\label{thm4}
		Let $s\in\mathbb{C}$ with $Re(s)<0$ and $\alpha\in\mathbb{R}_{>0}\backslash\mathbb{N},$ then we have
		\begin{equation}\label{eqn42}
			\zeta^{\alpha}(s,a)=2(2\pi)^{s-1}e^{\dot{\iota}\pi\alpha}\sum_{r=0}^{\infty} \sum_{l=0}^{\infty} \left(a_{r\alpha l}sin\left(\frac{\pi s}{2}\right)+b_{r\alpha l}cos\left(\frac{\pi s}{2}\right)\right)\Gamma^{(r)}(1-s),
		\end{equation}
		where the coefficients  $a_{r\alpha l}$ and $b_{r\alpha l}$ are given by
		$$a_{r\alpha l}=\sum_{k=0}^{\infty}\frac{A^{\alpha}_{r,k,l}}{log^{r+k+l-\alpha}(2\pi )}(-\frac{\pi}{2})^{k}\left[cos\left(\frac{\pi k}{2}\right)f(s)-sin\left(\frac{\pi k}{2}\right)h(s)\right],$$
		$$b_{r\alpha l}=\sum_{k=0}^{\infty}\frac{A^{\alpha}_{r,k,l}}{log^{r+k+l-\alpha}(2\pi )}(-\frac{\pi}{2})^{k}\left[sin\left(\frac{\pi k}{2}\right)f(s)+cos\left(\frac{\pi k}{2}\right)h(s)\right],$$
		and the functions $f(s)$ and $h(s)$ are given by
		$$f(s)=\sum_{q=1}^{\infty} \frac{(-logq)^{l}cos(2\pi qa)}{q^{1-s}},$$
		$$h(s)=\sum_{q=1}^{\infty} \frac{(-logq)^{l}sin(2\pi qa)}{q^{1-s}}.$$
		\begin{proof}
			The trigonometric identity
			$$sin\frac{\pi (s+k)}{2}=sin\frac{\pi s}{2}cos\frac{\pi k}{2}+cos\frac{\pi s}{2}sin\frac{\pi k}{2},$$
			$$cos\frac{\pi (s+k)}{2}=cos\frac{\pi s}{2}cos\frac{\pi k}{2}-sin\frac{\pi s}{2}sin\frac{\pi k}{2}$$
			when substituted in Equation $\ref{eqn23}$ gives Theorem $\ref{thm4}$.
		\end{proof}
	\end{thm}
	Now by Theorem $\ref{thm4}$, Equation $\ref{eqn9}$ which is Functional Equation of $\zeta^{\alpha}(s)$ can also be rewritten in terms of sines and cosines as follows:
	\begin{cor}
		Let $s\in\mathbb{C}$ with $Re(s)<0$ and $\alpha\in\mathbb{R}_{>0}\backslash\mathbb{N},$ then we have
		\begin{multline}\label{eqn43}
			\zeta^{\alpha}(s)=2(2\pi)^{s-1}e^{\dot{\iota}\pi\alpha}\sum_{r=0}^{\infty} \sum_{l=0}^{\infty} \left(a_{r\alpha l}sin\left(\frac{\pi s}{2}\right)+b_{r\alpha l}cos\left(\frac{\pi s}{2}\right)\right)\\\times\left[\zeta^{(l)}(1-s)\Gamma^{(r)}(1-s)\right],
		\end{multline}
		where the coefficients  $a_{r\alpha l}$ and $b_{r\alpha l}$ are given by
		$$a_{r\alpha l}=\sum_{k=0}^{\infty}\frac{A^{\alpha}_{r,k,l}}{log^{r+k+l-\alpha}(2\pi )}(-\frac{\pi}{2})^{k}cos\left(\frac{\pi k}{2}\right),$$
		$$b_{r\alpha l}=\sum_{k=0}^{\infty}\frac{A^{\alpha}_{r,k,l}}{log^{r+k+l-\alpha}(2\pi )}(-\frac{\pi}{2})^{r}sin\left(\frac{\pi k}{2}\right).$$
		\begin{proof}
			Putting $a=1$ in Theorem $\ref{thm4}$, we get the desired result as we know that $\zeta^{\alpha}(s,1)=\zeta^{\alpha}(s)$.
		\end{proof}
	\end{cor}
	\subsection{Functional Equation for $\theta^{\alpha}(s)$}\label{subsec4}
	\begin{thm}\label{thm5}
		Let $s\in\mathbb{C}$ with $Re(s)<0$ and $\alpha\in\mathbb{R}_{>0}\backslash\mathbb{N},$ then we have
		\begin{equation}\label{eqn44}
			\theta^{\alpha}(s)=\sum_{n\in\mathbb{Z}}e^{\dot{\iota}\pi n}(\pi n^{2})^{\alpha}e^{-\pi n^{2}s}
		\end{equation}
		\begin{proof}
			By letting $f(s)=\theta(s)$ in Equation $\ref{eqn7}$, we get
			\begin{align*}
				\theta^{\alpha}(s)
				&=\lim_{l \to 0^{+}} \frac{\sum_{m=0}^{\infty}\binom{\alpha}{m}(-1)^m \theta(s-ml)}{l^{\alpha}}\\
				&=\lim_{l \to 0^{+}} \frac{\sum_{m=0}^{\infty}\binom{\alpha}{m}(-1)^m \sum_{n\in\mathbb{Z}}e^{-\pi n^2(s-ml)}}{l^{\alpha}}\\
				&=\lim_{l \to 0^{+}} \frac{\sum_{n\in\mathbb{Z}}e^{-\pi n^2s}\sum_{m=0}^{\infty}\binom{\alpha}{m}(-1)^m e^{\pi n^2ml}}{l^{\alpha}}\\
				&=\sum_{n\in\mathbb{Z}}e^{\dot{\iota}\pi n}e^{-\pi n^2s}\lim_{l \to 0^{+}}\left(\frac{e^{\pi n^{2}l}-1}{l}\right)^{\alpha}\\
				&=\sum_{n\in\mathbb{Z}}e^{\dot{\iota}\pi n}(\pi n^{2})^{\alpha}e^{-\pi n^{2}s}
			\end{align*}	
		\end{proof}
	\end{thm}
	\begin{thm}\label{thm6}
		Let $s\in\mathbb{C}$ with $Re(s)<0$ and $\alpha\in\mathbb{R}_{>0}\backslash\mathbb{N},$ then we have
		\begin{equation}\label{eqn45}
			\theta^{\alpha}(s)=\sum_{k=0}^{\infty}\binom{\alpha}{k}\frac{e^{\dot{\iota}\pi\alpha}(2k+1)!!}{2^{k}(2k+1)s^{\frac{2\alpha+3}{2}}}\theta^{(\alpha-k)}\left(\frac{1}{s}\right),
		\end{equation}
		where $(2k+1)!!=(2k+1)(2k-1)(2k-3)......(1)$
		\begin{proof}
			From Equation $\ref{eqn12}$, we know that
			$$\theta(s)=\frac{1}{\sqrt{s}}\theta\left(\frac{1}{s}\right), Re(s)>0$$
			Now take the $\alpha$-order Fractional derivative of $\theta(s)$ and use Generalized Leibniz rule i.e. Equation $\ref{eqn8}$ to the above Equation's right hand side. This implies
			\begin{equation}\label{eqn46}
				\theta^{\alpha}(s)=\sum_{k=0}^{\infty}\binom{\alpha}{k}\left(\frac{1}{\sqrt{s}}\right)^{(k)}\theta^{(\alpha-k)}\left(\frac{1}{s}\right)
			\end{equation}
			Now we know that
			\begin{equation}\label{eqn47}
				\left(\frac{1}{\sqrt{s}}\right)^{(k)}=\frac{e^{\dot{\iota}\pi k}(2k+1)!!}{2^{k}(2k+1)s^{\frac{2k+1}{2}}},
			\end{equation}
			and
			\begin{equation}\label{eqn48}
				\theta\left(\frac{1}{s}\right)^{(\alpha-k)}=\frac{e^{\dot{\iota}\pi (\alpha-k)}}{s^{\alpha-k+1}}\theta^{(\alpha-k)}\left(\frac{1}{s}\right)
			\end{equation}
			By putting Equation $\ref{eqn47}$ and $\ref{eqn48}$ in Equation $\ref{eqn46}$, we get Theorem $\ref{thm6}.$
		\end{proof}
	\end{thm}
	\begin{rem}
		Now, let's focus only on the consistency of Equation $\ref{eqn45}$. Note that by Equations $\ref{eqn21}$ and $\ref{eqn22},$ 
		\begin{equation}\label{eqn49}
			g^{(\alpha)}(s)\rightarrow g^{(0)}(s)=g(s).
		\end{equation}
		as $\alpha\rightarrow0^{+}$.
		
		Specifically, all of Equation $\ref{eqn21}$ and $\ref{eqn22}$'s hypotheses are satisfied by the theta function, and thus 
		\begin{equation}\label{eqn50}
			\theta^{(\alpha)}(s)\rightarrow \theta(s).
		\end{equation}
		as $\alpha\rightarrow0^{+}$.
		
		Theorem $\ref{thm6}$'s proof can, of course, be read backwards until Equation $\ref{eqn46}$. As a result, Equation $\ref{eqn46}$'s right-hand side converges to that of Equation $\ref{eqn12}$ as $\alpha\rightarrow0^{+}$.
		Therefore, Equation $\ref{eqn45}$ agrees with the definition and formulation of theta function.
	\end{rem}
	\subsection{Relation between $\zeta^{\alpha}(s)$ and $\theta(t)$}\label{subsec5}
	\begin{thm}\label{thm7}
		Let $s\in\mathbb{C}$ with $Re(s)<0$ and $\alpha\in\mathbb{R}_{>0}\backslash\mathbb{N},$ then we have
		\begin{equation}\label{eqn51}
			\zeta^{\alpha}(s)=\sum_{k=0}^{\infty}\sum_{j=0}^{\infty}\frac{e^{\dot{\iota}\pi(k-j)}\alpha!log^{j}(\pi)(\pi)^{\frac{s}{2}}}{(\alpha-k)!j!2^{\alpha+1}\Gamma^{k-j+1}(\frac{s}{2})}\int_{0}^{\infty}(\theta(t)-1)t^{\frac{s}{2}}log^{\alpha-k}(t)\frac{dt}{t}
		\end{equation}	
		\begin{proof}
			From \cite{r12} we have
			\begin{equation}\label{eqn52}
				\zeta(s)=\frac{\pi^{\frac{s}{2}}}{2\Gamma(\frac{s}{2})}\int_{0}^{\infty}(\theta(t)-1)t^{\frac{s}{2}}\frac{dt}{t},s\in\mathbb{C}.
			\end{equation}
			Now take the $\alpha$-order Fractional derivative of $\zeta(s)$ and use Generalized Leibniz rule i.e. Equation $\ref{eqn8}$ to Right hand side of the above Equation.This implies
			\begin{equation}\label{eqn53}
				\zeta^{\alpha}(s)=\frac{1}{2}\sum_{k=0}^{\infty}\binom{\alpha}{k}\left({\frac{\pi^{\frac{s}{2}}}{\Gamma(\frac{s}{2})}}\right)^{(k)}\left(\int_{0}^{\infty}(\theta(t)-1)t^{\frac{s}{2}}\frac{dt}{t}\right)^{(\alpha-k)}
			\end{equation}
			Now we calculate $\left({\frac{\pi^{\frac{s}{2}}}{\Gamma(\frac{s}{2})}}\right)^{(k)}$ by applying Generalized Leibniz rule i.e. Equation $\ref{eqn8}$. This implies
			\begin{equation}\label{eqn54}
				\left({\frac{\pi^{\frac{s}{2}}}{\Gamma(\frac{s}{2})}}\right)^{(k)}=\sum_{j=0}^{\infty}\binom{k}{j}(\pi^{\frac{s}{2}})^{(j)}\left(\frac{1}{\Gamma(\frac{s}{2})}\right)^{(k-j)}=\sum_{j=0}^{\infty}\frac{e^{\dot{\iota}\pi(k-j)}k!log^{j}(\pi)(\pi)^{\frac{s}{2}}}{2^{k}j!\Gamma^{k-j+1}(\frac{s}{2})}
			\end{equation}
			Similarly
			\begin{equation}\label{eqn55}
				\left(\int_{0}^{\infty}(\theta(t)-1)t^{\frac{s}{2}}\frac{dt}{t}\right)^{(\alpha-k)}=\int_{0}^{\infty}(\theta(t)-1)t^{\frac{s}{2}}\frac{log^{\alpha-k}(t)}{2^{\alpha-k}}\frac{dt}{t}
			\end{equation}
			By putting Equations $\ref{eqn54}$ and $\ref{eqn55}$ in Equation $\ref{eqn53}$, we get Theorem $\ref{thm7}$.
		\end{proof}
	\end{thm}
	\begin{rem}
		Now, let's focus only on the consistency of Equation $\ref{eqn51}$. Note that by Equations $\ref{eqn21}$ and $\ref{eqn22},$ 
		\begin{equation}\label{eqn56}
			g^{(\alpha)}(s)\rightarrow g^{(0)}(s)=g(s).
		\end{equation}
		as $\alpha\rightarrow0^{+}$.
		
		Specifically, all of Equation $\ref{eqn21}$ and $\ref{eqn22}$'s hypotheses are satisfied by the Hurwitz Zeta function, and thus 
		\begin{equation}\label{eqn57}
			\zeta^{(\alpha)}(s,a)\rightarrow \zeta(s,a).
		\end{equation}
		as $\alpha\rightarrow0^{+}$.
		
		Theorem $\ref{thm7}$'s proof can, of course, be read backwards until Equation $\ref{eqn53}$. As a result, Equation $\ref{eqn53}$'s right-hand side converges to that of Equation $\ref{eqn52}$ as $\alpha\rightarrow0^{+}$.
		Therefore, Equation $\ref{eqn51}$ agrees with the definitions and formulations of Zeta functions.
	\end{rem}
	\section{Conclusions}
	This paper presents a comprehensive exploration of the properties of the Hurwitz Zeta function's fractional derivative, utilizing generalization to the Complex variable of the Grünwald-Letnikov Fractional derivative that adheres the Generalized Leibniz rule. The paper unveils a remarkable identity by iteratively applying the Generalized Leibniz rule, thereby extending the Hurwitz Zeta function's fractional functional equation to a Generalized form. Additionally, the study provides several Formulations and Simplifications of the Fractional Functional Equation of Hurwitz Zeta function, enriching our understanding of its intricate nature.
	
	In a captivating extension of the analysis, the paper delves into the Fractional derivative and Jacobi theta function's Functional Equation. Through meticulous derivation, it establishes a connection between Jacobi theta function's Functional Equation and the Fractional derivative. Furthermore, a fascinating relation between Riemann Zeta function's fractional derivative and the renowned Jacobi theta function is revealed. This connection opens up new avenues for exploration and deepens our understanding of the intricate relationships among these mathematical entities.
	\section{Acknowledgements}
	The authors are grateful to the editors for their careful reading of our manuscript. 
	\subsection{Statements and Declaration}
	\subsubsection{Funding} 
	There was no external funding for this study.
	\subsubsection{Competing Interests} 
	The authors declare that none of their known personal or financial conflicts could have possibly affected the research described in this work.
	\subsubsection{Author Contributions} 
	All authors made contributions to the conception and design of the study. [Ashish Mor] was in charge of gathering information, analysing it, and preparing the materials. All contributors offered input on earlier draughts of the manuscript after [Surbhi Gupta] produced the first draught and then [Manju Kashyap] approved the final draft. All authors read through and approved the final manuscript.
	\subsubsection{Data Availability} 
	No data required for this article.

\end{document}